\documentclass[12pt]{amsart}
\usepackage{amssymb}
\newcommand{\C}{\mathbb{C}}
\newcommand{\R}{\mathbb{R}}

\def\be{\begin{equation}}
\def\ee{\end{equation}}
\def\bequnan{\begin{eqnarray*}}
\def\eequnan{\end{eqnarray*}}
\def\la{\label}
\def\diff{\mathrm{d}}
\def\von{\varepsilon}

\def\P1{{\mathbf P}}

\def\cH{{\sf H}}
\def\cG{{\sf G}}

\def\wt{\widetilde}

\def\({\left(}
\def\Ran{\operatorname{Ran}}
\def\){\right)}
\def\s2{ \wt{S_2} }

\def\llangle{\left\langle}
\def\rrangle{\right\rangle}

\def\len{\left\|}
\def\rin{\right\|}
\def\yOonequ{{\stackrel{\lower.5mm\hbox{$ \scriptstyle\circ $}}y}_1^{\lower3pt\hbox{$ \scriptstyle 2 $}}}
\def\yOtwoqu{{\stackrel{\lower.5mm\hbox{$ \scriptstyle\circ $}}y}_2^{\lower3pt\hbox{$ \scriptstyle 2 $}}}
\def\yOonequti{\stackrel{\lower1mm\hbox{$ \scriptstyle\circ $}}{\widetilde y}_1^{\lower4pt\hbox{$ \scriptstyle 2 $}}}
\def\yOtwoquti{\stackrel{\lower1mm\hbox{$ \scriptstyle\circ $}}{ \widetilde y  }_2^{\lower4pt\hbox{$ \scriptstyle 2 $}} }
\def\yOone{{\stackrel{\lower.5mm\hbox{$ \scriptstyle\circ $}}y}_1}
\def\yOtwo{{\stackrel{\lower.5mm\hbox{$ \scriptstyle\circ $}}y}_2}
\def\yOoneti{\stackrel{\lower1mm\hbox{$ \scriptstyle\circ $}}{\widetilde y}_1}
\def\yOtwoti{\stackrel{\lower1mm\hbox{$ \scriptstyle\circ $}}{ \widetilde y  }_2 }

\begin{document}

\newcounter{prop}
\newtheorem{theorem}{Theorem}
\newtheorem{proposition}[prop]{Proposition}
\newtheorem{corollary}[prop]{Corollary}
\newtheorem{lemma}[prop]{Lemma}
\newcounter{rem}
\newtheorem{remark}[rem]{Remark}
\newtheorem*{remarknonumb}{Remark}
\newcounter{exerc}
\newtheorem{exercise}[exerc]{Exercise}
\def\theremark{\unskip}
\newtheorem{deffinition}{Definition}
\newtheorem{definition}{Definition}

\def\thedefinition{\unskip}


\begin{center}{\large ORDER PROBLEM FOR CANONICAL SYSTEMS AND \\ 
\vskip3mm
A CONJECTURE OF VALENT}
\end{center}
\bigskip \bigskip

\hskip2,5cm\vbox{\hsize10,5cm\baselineskip4mm
\noindent{\small\textbf{Abstract.} We establish a sharp upper estimate for the order of a canonical system in terms of the Hamiltonian. This upper estimate becomes an equality in the case of Krein strings. As an application we prove a conjecture of Valent about the order of a certain class of Jacobi matrices with polynomial coefficients.}

}


\bigskip

\centerline{\sf R.Romanov}

\medskip

\begin{center}
Department of Mathematical Physics and Laboratory of Quantum Networks, \\
Faculty of Physics, St Petersburg State University, \\
198504, St Petersburg, Russia,\\
e-mail: morovom@gmail.com
\end{center}

\bigskip \bigskip

\textbf{Introduction.} Let $ L $ be a positive number and $ \cH $  be a summable function on $ [ 0 , L ] $ with values in $ 2\times 2 $ matrices, such that $ \cH ( x ) \ge 0 $ a. e. Let  $ J= \left( \begin{array}{cc} 0 & -1 \cr 1 & 0 \end{array} \right) $. \textit{A canonical system} $ ( \cH , L ) $ is the matrix differential equation of the form
\be\la{can} J \frac{d Y}{d x } = z \cH Y ; \; z \in \C , \; x \in [ 0 , L ] . \ee
A solution $ M ( x , z ) $ of this equation satisfying $ M ( 0 , z ) = I $ is called the monodromy matrix.  We write $ M ( z ) = M ( L , z ) $. The function $ \cH $ is referred to as Hamiltonian. Without loss of generality we assume that $ \operatorname{tr} \cH ( x ) = 1 $ a. e. The background on canonical systems can be found in \cite{deBr, Sachnovich}.

Given a canonical system, the quantities 
\[ \limsup_{ |z| \to \infty } \frac{  \log | M_{ ij } (  z ) | }{ |z| }  \] and  
\[ \limsup_{ |z| \to \infty } \frac{  \log \log | M_{ ij } ( z ) | }{ \log |z| }  \]
do not depend on $ i , j $ (see e.g. \cite{BW}), and are called type and order of the system, respectively. The type is given by the classical Krein -- de Branges formula \cite{deBr},
\be\la{KdeB} \mbox{type of } ( \cH , L ) = \int_0^L \sqrt{ \det \cH ( t ) } \diff t . \ee  

In particular, this formula says that if $ \det \cH ( x ) = 0 $ a. e. then the matrix elements of $ M ( z ) $ have minimal type, and a fundamental question is to find or estimate the order of the system. This question is the order problem referred to in the title. One should notice here that the operators corresponding to canonical systems typically are not semibounded below hence conventional variational principles are not suitable for estimating their eigenvalues.

In the present paper we establish an upper estimate for the order in terms of the Hamiltonian which is sharp in the power scale and gives the actual value of the order in all available examples where it is known. Let us formulate the result.
 
\begin{definition}
A Hamitonian $ \cH $ is of finite rank if there exist numbers $ x_j $, $ 0 = x_0 < x_1 < \dots < x_N = L $, and a finite set of vectors, $ \{ e_j \}_{ j = 0 }^{ N-1 } $, $ e_j \in \R^2 $, of unit norm such that 
\[ \cH ( x )  = \llangle \cdot , e_j \rrangle_{ \C^2 }  e_j , \; x \in ( x_ j , x_{ j+1 } ) , \;  j = 0, \dots , N - 1 . \]
The (elements of) sets $ \{ x_j \} $,  $ \{ e_j \} $ and the number $ N $ are called parameters and the rank of the Hamiltonian,\footnote{Notice that we do not require $ e_j \ne e_{ j+1 } $, hence the rank of a finite rank Hamiltonian is not defined uniquely.} respectively.
\end{definition}
 
\begin{theorem} Let $ ( \cH , L ) $ be a canonical system and let $ 0 < d < 1 $.

1. Suppose that there exists a $ C > 0 $ such that for each $ R $ large enough there exists a Hamiltonian $ \cH_R $ of a finite rank, $ N(R) $, defined on $ (0 , L ) $ and a set of numbers (depending on $R$) $ \{ a_j \}_0^{ N(R) -1 } $, $ 0 < a_j \le 1 $, for which the following conditions are satisfied ($ P_j = \langle \cdot , e_j \rangle e_j $; $ x_j $, $ e_j $ are the parameters of $ \cH_R $),

(i) \[ \sum \frac 1{a_j^2} \int_{x_j }^{ x_{ j+1 }} \len \cH ( t ) - \cH_R ( t) \rin \diff t \le C R^{ d-1 } , \]

(ii) \[ \sum a_j^2 ( x_{ j+1 } - x_j ) \le C R^{ d-1 } , \]

(iii) \[ \sum \log \( 1 + \frac{ \len P_j - P_{ j+1 }  \rin }{ a_j a_{ j+1 } } \) \le C R^d , \]

(iv)
\[ \log a_0^{ -1 } +  \log a_{ N(R) -1 }^{ -1 }+ \sum \left| \log\frac{ a_j }{ a_{ j-1} } \right|\le C R^d . \]

Then there exists a $ K > 0 $ such that
\[ \len M ( z ) \rin \le e^{ K \left| z \right|^d } \]
for all $ z \in \C $. 

2. For each $ p $, $ 0 < p< 1 $, there exists a system $ ( \cH , L ) $ of order $ p $ which for any $ \von > 0 $ satisfies the assumption of assertion 1 with $ d = p + \von $.   
\end{theorem}

Theorem 1 gives an upper bound for the order in terms of the quality of approximation of the Hamiltonian by piecewise constants. The choice of approximators is natural in the sense that finite rank Hamiltonians have order zero (the corresponding monodromy matrices are polynomials), see also Section \ref{comont1}. Piecewise constant (or, more generally, polynomial) approximations are the mainstream in studying spectral asymptotics for integral and differential operators, see for instance \cite{BirmanS,BirmanS2}. By way of comparison, notice that those studies are mainly aimed at controlling the number of "pieces" necessary for approximation of a given function with a given accuracy, while in Theorem 1 the number $ N (R) $ does not play a direct role. In special situations, however, an optimal choice of approximation leads to conditions explicitly involving the number $ N (  R ) $ (see assumption (B) in the following theorem).   

An important class of canonical systems is constituted by systems with diagonal Hamiltonians $ \cH $. Such systems arise in description of mechanical strings with variable density sometimes called Krein strings, see \cite{KWW} for details. In the context of the order problem ($ \operatorname{rank} \cH ( x ) = 1 $) a diagonal Hamiltonian may take only two values, $ \frak H_ 1 = \begin{pmatrix} 1 & 0 \cr 0 & 0 \end{pmatrix} $ and $ \frak H_2 = \begin{pmatrix} 0 & 0 \cr 0 & 1 \end{pmatrix} $. Our next result says that in this case the upper bound implied by Theorem 1 coincides with the actual order. The formulation is as follows. Define $ X_1 = \{ x \in ( 0 , L ) \colon \cH ( x ) = \frak H_1 \} $, $ X_2 = \{ x \in ( 0 , L ) \colon \cH ( x ) = \frak H_2 \} $.  Let $ | \cdot | $ stand for the Lebesgue measure. 

\begin{theorem}\la{selfsim}
Suppose that for a. e. $ x\in [ 0 , L ] $ either $ \cH ( x ) = \frak H_ 1 $, or $ \cH ( x ) = \frak H_2 $. Then the order of the system $ ( \cH, L ) $ coincides with the infimum of $ d $'s , $ 0< d < 1 $, for which there exists a positive $ C = C ( d ) $ such that for each $ R $ large enough there exist a covering of the interval $ ( 0 , L ) $ by $ n = n ( R ) $ intervals, $ \omega_j $, such that

\textit{(A)} \[ \sum \sqrt{ | \omega_j \cap X_1 | \, | \omega_j \cap X_2 | } \le C R^{ d-1 }  ; \] 

\textit{(B)} \[ n( R ) \le C R^d . \]
\end{theorem} 

We give several examples of application of Theorems 1 and 2. Namely, we establish upper estimates of the order in terms of smoothness for Hamiltonians from classical smoothness classes (H\"older, bounded variation) by applying Theorem 1, see Corollary \ref{var}, prove a conjecture of Valent about order of a certain class of Jacobi matrices (Corollary \ref{Valenthyp}), and give a rather short calculation of the order for the Cantor string (see Section \ref{Castr}).

The first result on the order problem we are aware of is the 1939 theorem of Liv\v sic \cite{Livshitz} saying that the order of a canonical system corresponding to an indeterminate moment problem with moments $ \gamma_j $ is not less than $ \limsup_{ n \to \infty } (2n \log n ) /\log \gamma_{ 2n}  $, the order of the entire function $ \sum z^{ 2j }/ \gamma_{ 2j } $. A modern two-line proof of this assertion can be found in \cite{BergSzwarc}. 
The next result, essentially due to Berezanski\u{\i} \cite{Berez}, is formulated in terms of Jacobi matrices. Berezanski\u{\i} studied Jacobi matrices of the form 
\be\la{Jac} \( \begin{array}{cccccc} q_1 & \rho_1 & 0 & \dots & & \cr 
\rho_1 & q_2 & \rho_2 & 0 & \dots & \cr
0 & \rho_2 & q_3 & \rho_3 & 0 & \dots \cr  & 0 & \ddots & \ddots & \ddots & \ddots \end{array} \) \ee 
with $ \rho_j$ growing. Although he did not explicitly address the problem of order, he made a crucial technical observation that allows to estimate the corresponding orthogonal polynomials at large $ j $ in terms of $ \rho_j^{ -1 } $. The explicit translation of his result to the order problem is given in \cite{BergSzwarc}. It says essentially that if $ \rho_j $ is a log-convex or log-concave sequence at large $ j $, and $ q_j $ is small relative to $ \rho_j $, then the order of the system coincides with the convergence exponent for the sequence $ \rho_j $. More precisely, the following assertion holds.

\begin{theorem}\la{Bsw}\cite{BergSzwarc} Let (\ref{Jac}) be a limit-circle Jacobi matrix. If $ \rho_j $ satisfies $ \rho_{ j- 1} \rho_{ j+1 } \le \rho_j^2 $ for all $ j $ large enough, and $ q_j / \rho_{ j-1 } \in l^1 $, then the order of the system equals to $ \inf \{ \alpha > 0  \colon \rho_j^{ - \alpha } \in l^1 \} $. The same result holds if $ \rho_j $ satisfy the inequality $ \rho_{ j- 1} \rho_{ j+1 } \ge \rho_j^2 $ instead.
\end{theorem}

In the language of canonical systems, these results refer to a special class of Hamiltonians defined as follows. Let $ b_j $ be a bounded sequence of reals, $ 0 = b_0 < b_1 < b_2 < \dots $, $ L = \lim b_j $, and $ e_j \in {\mathbb{R}}^2 $, $ j \ge 1 $, a sequence of vectors of unit norm, $ e_j \ne \pm e_{j-1} $. Let $ \Delta_j = ( b_{ j-1 } , b_j ) $, $ j \ge 1 $. 
Define the Hamiltonian $ \cH $ on $ ( 0 , L ) $ corresponding to these sequences by  
\be\la{HamJacobi} \cH ( x ) = \langle \cdot , e_j \rangle e_j , \; \; x \in \Delta_ j . \ee
The correspondence between Hamiltonians of this form and limit-circle Jacobi matrices is described in detail in \cite{Katz}. Upon suitable normalization it is one-to-one, the corresponding selfadjoint operators are unitarily equivalent, and the Jacobi parameters $ q_j $, $ \rho_j $ are expressed via $ e_j $ and $ b_j $ by explicit formulae. The relation of Theorem \ref{Bsw} and our result is that the relevant part of Theorem \ref{Bsw} (the order is not greater than the convergence exponent) easily follows from Theorem 1 applied to Hamiltonians of this class, see Section \ref{comparis} for details.

Apart from the mentioned general results, there are several isolated explicitly solvable non-trivial examples of Jacobi matrices for which the order is known and is non-zero. Two of them were found by Valent and his collaborators in  \cite{BergValent} (order $ 1/4 $) and \cite{GLV} (order $ 1/3 $). In these examples, motivated by studies of the birth-death processes, $ q_n $ and $ \rho_n^2 $ are polynomials, $ | q_n | \sim \rho_n $ at infinity. On their basis it was conjectured in \cite{Valent} that in a class of Jacobi matrices with polynomial $ \rho_n^2 $ and $ q_n $ the order is $1/\deg q_n $. As explained below, the fact that the order is not less than $ 1 / \deg q_n $ is almost trivial, hence the hypothesis is essentially that the order does not exceed $ 1 / \deg q_n $. We establish the latter in Corollary \ref{Valenthyp} applying Theorem 1. 

Another set of examples in the order problem comes from studies of non-Weyl spectral asymptotics for 1D differential operators. Apparently the first result in this direction is due to Uno and Hong \cite{UH} who have found the order for the Cantor string. In \cite{SolVerb},  the authors calculated the order (in fact, they found, in a sense, the whole leading term), for a class of strings with self-similar weights. The order is also known for a rather general class of strings related to so called $ d $--sets \cite{Triebel}. Notice that \cite{UH,SolVerb} rely on the variational prinicple for the eigenvalues, hence their methods are apparently unsuitable to obtain results like Theorem 1 because of lack of semiboundedness. 

A general formula for the order of a string was obtained in \cite{Katz1}. In the situation of Theorem 2 it says that the order of the system $ ( \cH , L) $ is 
\be\la{Katzformula} \inf\left\{ d> 0 \colon \int_0^{ \tilde L } dM( x) \int_0^{ \min \{ x , \tilde L-x \} } \( s ( M ( x+s ) - M ( x -s ) ) \)^{ \frac d2 - 1} \diff s < \infty \right\} .\ee
Here $ M $ is a non-decreasing singular function on an interval $ [ 0 , \tilde L ] $, $ \tilde L + M ( \tilde L ) = L $, such that $ X_1 = \{ x + M ( x ) \colon x \in [ 0 , \tilde L ], \, M^\prime ( x ) = 0 \} $.  
This formula, to the best of our knowledge, has never been used to calculate the order of an actual string of the class considered in this paper (see also Section \ref{Katzf}). We use an argument from the proof of (\ref{Katzformula}) in \cite{Katz1} in the derivation of Theorem 2, see  Lemma \ref{Katzform}.   

The structure of the paper is as follows. In Section \ref{deBrK} we reproduce a proof of the inequality "lhs of (\ref{KdeB})"$ \le $ "rhs of  (\ref{KdeB})" from \cite{deBrII} with a minor simplification. The reason we give it here is that it provides one of the ideas used in the proof of Theorem 1. The proofs of Theorems 1 and 2 occupy  sections named accordingly. In the Comments sections we discuss the assumptions of these theorems and compare them with the earlier results. In the Applications section we establish upper bounds for the order in smooth classes and prove the Valent conjecture.

Throughout the paper the norm signs refer to the operator norm for $ 2 \times 2 $ matrices, $ \frak H_{ 1,2 } $ are the matrices defined before Theorem 2. Unless specified otherwise summations extend to all values of the summation parameter for which the summand is defined. $ C $ stands for any constant whose exact value is of no interest for us. Given a Jacobi matrix (\ref{Jac}), $ P_j ( \lambda ) $ and $ Q_j (\lambda) $ stand for the solutions of the corresponding three-term requrrency relation subject to the initial conditions $ P_1 = 1 $, $ P_0 = 0 $, $ Q_1 = 0 $, $ Q_2 = 1/\rho_1 $ (orthogonal polynomials of the first and second kind, respectively). 

\section{The upper estimate in the Krein--de Branges formula}\la{deBrK}

\begin{proposition} Let $ ( \cH , L ) $ be a canonical system. Then 
\[ \mbox{type of } ( \cH , L ) \le \int_0^L \sqrt{ \det \cH ( t ) } \diff t . \]
\end{proposition} 

Let $ p ( x ) $ be the exponential type of $ M ( x , \lambda ) $. For each $ y \in ( 0 , L ) $ the monodromy matrix satisfies the integral equation
\be\la{intmono} M ( x , \lambda ) = M ( y , \lambda ) - \lambda \int_y^x J \cH ( t ) M ( t , \lambda ) \diff t . \ee 
A crude estimate of the Volterra iterations for this equation shows that $ | p ( x ) - p ( y ) | \le | x - y | $ and thus $ p ( x ) $ is Lipschitz. The idea of the proof is to estimate $ p^\prime (x ) $ in terms of $ \cH $ and then integrate it to obtain the required bound. 

\begin{proof} Let $ \Omega $ be a constant invertible matrix, to be chosen later. Equation (\ref{intmono}) can then be rewritten as follows,
\[ \Omega M ( x , \lambda ) = \Omega M ( y , \lambda ) - \lambda \int_y^x \( \Omega J \cH ( t ) \Omega^{ -1 } \) \Omega M ( t , \lambda ) \diff t . \]
This is a Volterra equation with respect to $ \Omega M $. Solving it by iterations we have (the Gronwall lemma),
\be\la{estGron} \len \Omega M ( x , \lambda ) \rin \le  \len \Omega M ( y , \lambda ) \rin \exp \( \left| \lambda \right| \int_y^x \len \Omega J \cH ( t ) \Omega^{ -1 } \rin \diff t \) \ee
for all $ y \le x $. It follows that $ p ( x ) $ satisfies 
\[ p ( x ) \le p ( y ) +   \int_y^x \len \Omega J \cH ( t ) \Omega^{ -1 } \rin \diff t . \] 
Taking the limit $ y \uparrow x $ we obtain that for a. e. $ x \in [ 0 , L ] $
\[ p^\prime ( x ) \le  \len \Omega J \cH ( x ) \Omega^{ -1 } \rin . \]
The lhs does not depend on $ \Omega $, hence let us minimize the rhs in $ \Omega $.

\begin{lemma}\la{variat} 
Let $ A $ be a $ 2 \times 2 $-matrix with $ \operatorname{tr} A = 0 $, then 
\be\la{zerotr} \inf_{ \Omega \colon \det \Omega \ne 0 } \len \Omega A \Omega^{ -1 } \rin = \sqrt{ | \det A | } . \ee
\end{lemma} 
\begin{proof} 
If $ \det A \ne 0 $ the lemma is trivial -- it suffices to choose $ \Omega $ to be the diagonalizer of $ A $. If $ \det A = 0 $, then without loss of generality one can take $ A = \begin{pmatrix} 0 & 0 \cr 1 & 0 \end{pmatrix} $, $ \Omega = \mbox{diag} \( a^{ -1 } , a \) $, and send $ a \to 0 $.
\end{proof}

Applying this lemma gives $ p^\prime ( x ) \le \sqrt{ \det \cH ( x ) } $ a. e., and the assertion follows by integrating this inequality. \end{proof}

This proof is essentially the one in \cite[Theorem X]{deBrII} except that de Branges uses an explicit reduction of the matrix $ J \cH ( x) $ rather than mere existence of a diagonalizer.

\section{Proof of Theorem 1}

\subsection{The estimate.} Let $ ( \cH , L ) $ be a canonical system. For an arbitrary finite set of numbers $ x_j $, $ 0 \le j \le N $, such that $ 0 = x_0 < x_1 < x_2 < \dots < x_N = L $, and arbitrary invertible matrices $ \Omega_j $, $ 0 \le j \le N $, an argument from the proof of Proposition 1 (consider (\ref{estGron}) with $ y = x_j $, $ x = x_{ j+1 } $, $ \Omega = \Omega_j $) shows that
\[ \len \Omega_j  M ( x_{j+1} , \lambda ) \rin \le  \len \Omega_j M ( x_j , \lambda ) \rin \exp \( \left| \lambda \right| \int_{ x_j}^{x_{j+1}} \len \Omega_j J \cH ( t ) \Omega_j^{ -1 } \rin \diff t \)  \]
for $ 0 \le j < N $. 
With the notation $ M_j = M ( x_j , \lambda ) $ we then have
\bequnan \len \Omega_{ j+1 } M_{ j+1 } \rin \le \len \Omega_{ j+1 } \Omega_j^{ -1 } \rin \cdot \len \Omega_j M_{j+1} \rin \le \len \Omega_{ j+1 } \Omega_j^{ -1 } \rin \len \Omega_j M_j \rin \\ \exp \( | \lambda |  \int_{ x_j}^{x_{ j+1 }} \len \Omega_j J \cH ( t ) \Omega_j^{ -1 } \rin \diff t \)  . \eequnan
Taking logarithm, summing the resulting inequalities in $ j $ and choosing $ \Omega_N = I $ we get,
\be\la{norm} \log \len M ( \lambda ) \rin \le | \lambda |  \sum_{j = 0 }^{ N-1 } \int_{ x_j}^{x_{ j+1 }} \len \Omega_j J \cH ( t ) \Omega_j^{ -1 } \rin \diff t + \sum_{ j=0 }^{ N-1 } \log \len  \Omega_{ j+1 } \Omega_j^{ -1 } \rin + \log \len \Omega_0 \rin . \ee
Let $ P_j $ be an orthogonal rank 1 projection, $ P_j = \langle \cdot , e_j \rangle e_j $, $ e_j \in \R^2 $, $\| e_j \| = 1 $. Then the summand in the first sum in the rhs can be estimated as follows,
\begin{eqnarray}\la{OmegaH} \int_{ x_j}^{x_{ j+1 }} \len \Omega_j J \cH ( t ) \Omega_j^{ -1 } \rin \diff t \le \int_{ x_j}^{x_{ j+1 }} \len \Omega_j J \( \cH ( t ) - P_j \) \Omega_j^{ -1 } \rin \diff t + ( x_{ j+1} - x_j ) \len \Omega_j J P_j \Omega_j^{ -1 } \rin \nonumber \\ \le \len \Omega_j \rin  \len \Omega_j^{ -1 } \rin \int_{ x_j}^{x_{ j+1 }} \len \cH ( t ) - P_j \rin \diff t + ( x_{ j+1} - x_j ) \len \Omega_j J P_j \Omega_j^{ -1 } \rin . \end{eqnarray}
Let us choose the matrices $ \Omega_j $. The choice is suggested by the proof of Lemma \ref{variat},
\[ \Omega_j = \operatorname{diag} \( a^{ -1 }_j , a_j \) U_j , \]
where $ U_j $ is a unitary transform reducing $ J P_j $ into its Jordan form, 
\[ U_j J P_j U_j^{ -1 } = \begin{pmatrix} 0 & 0 \cr 1 & 0 \end{pmatrix} , \] 
and $ a_j \in ( 0 , 1 ] $. More precisely, we set $ U_j = e^{ - \varphi_j J } $ with $ \varphi_j \in [ 0 , 2 \pi ) $ defined by $ e_j = \begin{pmatrix} \cos \varphi_j \cr \sin \varphi_j \end{pmatrix} $. With this choice
 
$ 1^\circ $. $ \len \Omega_j \rin = \len \Omega_j^{ -1 } \rin =  a_j^{ -1 } $, $ \len  \Omega_j J P_j \Omega_j^{ -1 } \rin = a_j^2 $, and one can continue the estimate (\ref{OmegaH}),
\be\la{OmegaH1} \textrm{rhs of } (\ref{OmegaH}) \le \frac 1{a_j^2} \int_{ x_j}^{x_{ j+1 }} \len \( \cH ( t ) - P_j \) \rin \diff t + a_j^2 \( x_{ j+1 } - x_j \)  . \ee

$ 2^\circ $. For $ j \le N-2 $ 
\bequnan \Omega_{ j+1 } \Omega_j^{ - 1 } & = &\begin{pmatrix} a_{ j+ 1 }^{ -1 } & 0 \cr 0 & a_{j+1} \end{pmatrix} U_{ j+1 } U_j^{ -1 }  \begin{pmatrix} a_j & 0 \cr 0 & a_j^{ - 1 } \end{pmatrix} = \\ & & \begin{pmatrix} a_{ j+ 1 }^{ -1 } & 0 \cr 0 & a_{ j+1 } \end{pmatrix} e^{ \( \varphi_j - \varphi_{ j+1 } \) J } \begin{pmatrix} a_j & 0 \cr 0 & a_j^{ - 1 } \end{pmatrix} = \\ & & \begin{pmatrix} a_j  a_{ j+ 1 }^{ -1 } & 0 \cr 0 & a_{j+1}  a_j^{ -1 }  \end{pmatrix}  + O \( \frac{ \| P_{ j+1 } - P_j \| }{ a_j a_{ j+1 } } \) . \eequnan
Then ($ \log ( x +y ) \le | \log y | + \log ( 1 + x ) $ for $ x , y > 0 $)
\[ \log \len  \Omega_{ j+1 } \Omega_j^{ - 1 } \rin \le  \left| \log \( a_j a_{ j+ 1 }^{ -1 } \) \right| +  C \log \( 1+  \frac{ \| P_{ j+1 } - P_j \| }{ a_j a_{ j+1 } } \) . \]
Plugging this and (\ref{OmegaH1}) in (\ref{norm}) and taking into account that $ \len \Omega_0 \rin = a_0^{ -1 } $,  $ \len \Omega_{ N-1 }^{ -1 } \rin = a_{ N-1 }^{ -1 } $, we obtain the first assertion of the theorem.

\subsection{Sharpness.} Let $ p \in ( 0 , 1 ) $, $ \alpha = p^{ -1 } - 1 $, $ d = p + \von $. Define $ b_j = 1 - j^{ -\alpha } $ for $ j \ge 1 $, and for $ x \in [ 0 , 1 ] $ let
\[ \cH ( x ) = \begin{cases} \frak H_1 ,\; x \in \bigcup_j [ b_{ 2j-1 } , b_{ 2j } ]  \cr 
\frak H_2 ,\;  x \notin \bigcup_j [ b_{ 2j-1 } , b_{ 2j } ]  . \end{cases} \]

The required assertion will be proved if we show that (a) $ \cH $ satisfies the conditions (i)--(iv) of the theorem for all $ \von > 0 $, (b) the order of the system $ ( \cH , 1 ) $ is not less than $ p $. Given an $ R> 0 $, define 
\[ \cH_R ( x ) = \begin{cases} \frak \cH ( x ) ,\; x \in [0 ,  b_{ N-1}] \cr 
\frak H_1 , \;  x \in [ b_{ N-1 } , 1 ]  , \end{cases} \] with $ N = N( R) $ to be chosen later. Let $ a_{ N-1 } = 1 $, $ a_j = R^{ \frac { d-1 }2 } $ for $ j \le N-2 $. We then have,
\bequnan \textrm{lhs of (i)} & \le & \frac 2{ a_{ N-1 }^2  \( N-1 \)^\alpha } = O \( N^{ - \alpha } \) , \\
\textrm{lhs of (ii)} & \le & C \sum_0^{ N-2 } a_j^2 j^{ -1-\alpha } + N^{ -\alpha } = O \( R^{ d-1} \) + O \( N^{ - \alpha } \) , \\
\textrm{lhs of (iii)} & \le & C ( N - 1 ) \log R + O ( \log R ) = O \( N \log R \) , \\
\textrm{lhs of (iv)} & = &  O ( \log R ) . 
\eequnan 
Let $ N \sim R^p $ as $ R \to \infty $. Then assumptions (i)--(iv) are satisfied.

Let us now establish that the order of the system is not less than $ p $. To this end, we use the following identity. Let $ \Theta ( x , \lambda ) $ be the first column of $ M ( x , \lambda ) $. Differentiating $ \llangle \Theta , J \Theta \rrangle_{ \C^2 } $ with respect to the equation (\ref{can}), we find
\be\la{identM11} 
\Im \( M_{11} ( 1 , \lambda ) \overline{M_{21} \( 1 , \lambda \) } \) = \Im  \lambda \int_0^1 \llangle \cH ( t ) \Theta ( t , \lambda ) , \Theta ( t , \lambda ) \rrangle_{ \C^2 }  \diff t . \ee 

In the situation under consideration the rhs is 
\[ \Im \lambda \sum \( b_j - b_{j-1} \) \left| \begin{array}{cc} M_{11} ( b_j , \lambda ) , & j \textrm{ even} \cr M_{21} ( b_j , \lambda ) , & j \textrm{ odd} \end{array} \right|^2  . \]

For $ \Im \lambda > 0 $ one can estimate this quantity from below. Let $ \delta_j = b_j - b_{ j-1 } $. We have
\be\la{estM11} \textrm{rhs of (\ref{identM11})} \ge \Im \lambda  \sum \delta_{ 2j } \left| M_{11} ( b_{2j} , \lambda ) \right|^2  = \Im \lambda  \sum \delta_{ 2j } \left| M_{11} ( b_{2j-1} , \lambda ) \right|^2 . \ee
In the last equality we took into account that $ M_{11} ( b_{ 2j } , \lambda ) = M_{11} ( b_{ 2j -1 } , \lambda ) $ in the situation under consideration as $ M_{11}^\prime ( x , \lambda ) = 0 $ when $ x \in ( b_{ 2j-1 } , b_{ 2j } ) $.

To estimate the rhs of (\ref{estM11}) from below we use the following corollary of the fact that all the zeroes of matrix elements of $ M ( x , \lambda ) $ are real. 

\begin{remark} Let $ ( G , L ) $ be a canonical system such that $ ( 0 , L ) $ is a union of disjoint intervals, $ I_j $, accumulating only at $ L $, and $ G ( x ) $ is a constant rank 1 operator on each $ I_j $. Then $ M ( x , \cdot ) $ is a polynomial for all $ x \in ( 0 , L ) $. For any $ m, l $, $ 1 \le m,l \le 2 $, define $ k ( x ) $ to be the degree of the polynomial $ M_{ml} ( x , \cdot ) $, $ c ( x ) $ be its leading coefficient,
\[ M_{ ml} ( x , \lambda ) = c ( x ) \lambda^{ k ( x ) } + ( \textrm{a polynomial of degree} \le  k ( x ) - 1 ) .  \]
Then $ | M_{ ml} ( x , i \tau ) | \ge | c ( x ) | \tau^{ k ( x ) } $ when $ \tau > 1 $.
\end{remark} 

Applied in the situation under consideration to the left upper entry of the monodromy matrix at $ x = b_{ 2j-1 } $, this gives $ | M_{11} ( b_{ 2j-1 }, i \tau ) | \ge \left| c_j \right| \tau^{ k_j }$ for $ \tau > 1 $, $ c_j $ and $ k_j $ being the leading coefficient and the degree of the polynomial $ M_{11} ( b_{ 2j -1 } , \cdot ) $, resp., $ j \ge 1 $.

Let us calculate $ k_j $ and $ c_j $.  Define $ M_j ( \lambda ) $ to be the value of the matrix solution of (\ref{can}) with the Cauchy data $ Y ( b_j ) = I $ at $ x = b_{ j+1 } $, then 
\be\la{multt} M ( b_{ 2j -1 } , \lambda ) = M_{ 2j-2 } ( \lambda ) M_{ 2j-3 } ( \lambda ) \cdots M_2 ( \lambda ) M_1 ( \lambda ) \ee
by the multiplicative property of the monodromy matrices. A straightforward calculation gives 
\[ M_ j ( \lambda ) =  I  + \lambda \delta_{j+1} \left\{  \begin{array}{cc} \begin{pmatrix} 0 & 0 \cr -1 & 0 \end{pmatrix} , & j \textrm{ odd} \cr 
\begin{pmatrix} 0 & 1\cr 0 & 0 \end{pmatrix} , &  j \textrm{ even} . \end{array}  \right. \]  
The leading term in the matrix polynomial $ M ( b_{ 2j-1 } , \lambda ) $ comes from choosing the terms of the first order in $ \lambda $ in each multiple in (\ref{multt}). It has the form
\[ \( -1 \)^{j+1} \delta_{ 2j - 1 } \delta_{ 2j-2 } \cdots \delta_2 \begin{pmatrix} 1 & 0 \cr 0 & 0 \end{pmatrix} . \]
Thus, $ k_j = 2j-2 $, $ | c_j | = \prod_{ n = 2 }^{ 2j-1 } \delta_n $, and one can continue the inequality in (\ref{estM11}),
\[ \textrm{rhs of } (\ref{estM11}) \ge \sum \delta_{ 2j }  \( \prod_{ n = 2}^{ 2j-1 } \delta_n^2 \) \left| \lambda \right|^{ 2 ( 2j - 2 ) } \]
On the other hand, if $ \rho $ is the order of the system, then for any $ \von > 0 $ the lhs in (\ref{identM11})  is not greater than $ \exp \( C_\von  r^{ \rho + \von } \) $. By a standard relation between Taylor coefficients and exponential order (s. f. \cite{Levin}) it follows that for large $ j $
\[ \delta_{ 2j }  \( \prod_{ n = 2}^{ 2j-1 } \delta_n^2 \) \le \( \frac Cj \)^{ \frac {4j}{\rho + \von } } . \]
Notice that the sequence $ \delta_j = j^{ - \alpha } - \( j+1\)^{ -\alpha }  $ is monotone decreasing in $ j $, hence the last inequality implies that $ \delta_{ 2j } = O \( j^{ - 1/ ( \rho + \von ) } \) $. Comparing this with $ \delta_j \asymp j^{ - 1 - \alpha } = j^{ -1/p } $, we find $ \rho + \von \ge p $ for all $ \von > 0 $, that is, $ \rho \ge p $. The proof of Theorem 1 is thus completed.

\section{Comments on Theorem 1}

\subsection{Choice of approximants}\la{comont1} Let us consider a canonical system having Hamiltonian of the form $ \cH = \langle \cdot , e ( x ) \rangle e ( x ) $ with $ e ( x ) = \begin{pmatrix} u ( x ) \cr v( x) \end{pmatrix} $, $ u , v $ being smooth functions on $ ( 0 , L ) $ subject to the condition $ u^\prime v - v^\prime u = -1 $. The order of the canonical system with this Hamiltonian is $ 1/2 $, for if $ Y = \begin{pmatrix} Y_+ \cr Y_- \end{pmatrix} $ is a solution of the system, then a straightforward calculation \cite{Remling} shows that $ y = Y_+ u + Y_- v $ satisfies the Schr\"odinger equaiton $ - y^{\prime \prime } + q y = \lambda y $ with the potential $ q = u^{ \prime \prime } / u  $. This suggests that smooth functions cannot be used as approximants, at least in the whole range $ ( 0 , 1 ) $ of orders.

\subsection{Formulation} The Krein -- de Branges formula implies that if assumptions of Theorem 1 are satisfied for some $ d < 1 $ then $ \operatorname{rank} \cH ( x ) = 1 $ a. e. This fact can easily be seen directly.  Indeed, suppose that $ ( \cH , L ) $ is a canonical system such that for any $ R $ large enough the conditions (i)--(iv) are satisfied for some $ d \in ( 0 , 1 ) $. Define $ S_\epsilon = \{ t \colon \| \cH ( t ) f \| \ge \epsilon \| f \| \textrm{ for all } f \in \C^2 \} $, $ \epsilon > 0 $. Arguing by contradiction, let $ \epsilon > 0 $ be such that $ \left| S_\epsilon \right| > 0 $. Applying the Schwarz inequality and using conditions (i) and (ii) we find that $ \sum \sqrt{( x_{ j+1 } - x_j ) \int_{x_j }^{ x_{ j+1 }} \len \cH ( t ) - \cH_R ( t) \rin \diff t } \le C R^{ d-1 } $. The quantity $ \left| [ x_j , x_{ j+1 }] \cap S_\epsilon \right| $ estimates from below both factors in the summand, hence the sum is not less than $ \sum \left| [ x_j , x_{ j+1 }] \cap S_\epsilon \right| = \left| S_\epsilon \right| > 0 $. Taking the limit $ R \to \infty $ we obtain a contradiction.

\subsection{Sharpness} In the example establishing part 2 of Theorem 1 the conditions (i)--(iv) are satisfied with $ d = p $ if we insert the $ \log R $ factor in the rhs. It is not known to the author if one can get rid of the logarithmic factor, that is, part 2 holds with $ \varepsilon = 0 $.  

\subsection{Comparison with Theorem \ref{Bsw}}\la{comparis} The relevant part of Theorem 3 is the assertion that under the stated assumptions if $ \{ \rho_j^{ -1 } \} \in l^\alpha $ then the order $ \le \alpha $. Let us first translate the setup of Theorem 3 to the language of canonical systems. As mentioned in the introduction, given a Jacobi matrix (\ref{Jac}), explicit formulae that express the Jacobi parameters $ q_j $, $ \rho_j $ via the corresponding Hamiltonian are known, see \cite{Katz}. Let us reproduce them in a convenient form. In the following theorem $ P_n $ and $ Q_n $ are the orthogonal polynomials of first and second kind defined by the matrix (\ref{Jac}), respectively, and $ e_j $ and $ b_j $ are parameters of the corresponding Hamiltonian of the type described after Theorem 3. We write $ \varphi_j $ for the argument of the vector $ e_j \in \R^2 $, $ \Delta_j = ( b_{ j-1 } , b_j ) $, $ \delta_j = | \Delta_j | $. 

\begin{theorem}\cite{Katz}\la{Ka}
The correspondence between canonical systems and limit-circle Jacobi matrices can be chosen so that

(i) $  \delta_n = P_n \( 0 \)^2 + Q_n \( 0 \)^2 $,

(ii) \[ \rho_j =  \frac 1{| \sin \( \varphi_j - \varphi_{ j+1 } \) |  \sqrt{ 
\delta_{ j+1} \delta_j }}, \; j \ge 1 . \]
\end{theorem}

Now let $ q_j = 0 $ for simplicity. Then it can be shown \cite{Katz} that $ e_j \perp e_{ j-1 } $, $ e_1 = \begin{pmatrix} 1 \cr 0 \end{pmatrix} $, hence $ \rho_j^{ -1 } = \sqrt{ \delta_j \delta_{ j+1 } } $. Solving for $ \delta_j $ we have 
\[ \delta_{ j+1 } = \( \frac{ \rho_{ j-1 } \rho_{ j-3 } \cdots }{ \rho_j \rho_{ j-2 } \cdots } \)^2 . \] 
Let $ \rho_j $ satisfy the assumption of Theorem \ref{Bsw}. Then $ \rho_{ j-1 } / \rho_j $ at large $ j $ is a monotone sequence having a limit $ \le 1 $, and if the limit is $ 1$, then it is increasing. This implies that $ \delta_j = O \(  \rho_{ j-1 }^{ -1 } \) $ and therefore $ \{ \delta_j \} \in l^d $ for any $ d $ greater than the convergence exponent of the sequence $ \rho_j $. Notice that the mondromy matrix corresponding to the interval $ \Delta_j $ is $ I + O \( |\lambda | \delta_j \)  $. This, the multiplicative property of monodromy matrices and $ \{ \delta_j \} \in l^d $ imply by elementary inequalities that the monodromy matrix of the system in question is $ O \( e^{ C \left| \lambda \right|^d } \) $, which gives the assertion under consideration. 

It is just as easy to derive this assertion from Theorem 1. Let $ d $ be such that $ \{ \rho_j^{ -1 } \} \in l^d $ and  let $ \frak N = \{ j \colon  \delta_j > R^{ -1 } \} $. Define $ \cH_R (x ) =  \cH ( x ) $ whenever $ x \in \Delta_k $, $ k \in \frak N $. On the complement of $ \cup_{ k\in \frak N} \Delta_k $ we define $ \cH_R $ to be an arbitrary constant rank 1 orthogonal projection. With this definition, $ \cH_R $ is a finite rank Hamiltonian with parameters to be denoted $ x_j , f_j $. Let $ a_j = R^{ (d-1)/2 } $ whenever $ j $ is such that $ [ x_j , x_{ j+1 } ] $ coincides with one of the intervals $ \Delta_k $, $ k \in \frak N $, $ a_j =1 $ otherwise. Then conditions (i) and (ii) in Theorem 1 are satisfied because $ \{ \delta_j \} \in l^d $ and so $ \sum_{ \delta_j \le R^{ -1 } } \delta_j  = O \( R^{ d-1 } \) $, conditions (iii) and (iv) are satisfied because the number of $ j $'s for which $ \delta_j > R^{ -1 } $ is $ O \( R^d \) $, again by $ \{ \delta_j \} \in l^d $, and therefore the rank of $ \cH_R $ is $ O \( R^d \) $ as well. Applying Theorem 1 we conclude that under the assumptions of Theorem \ref{Bsw} with $ q_j = 0 $ the order is not greater than $ d$. Thus, our result generalizes the upper estimate in Theorem \ref{Bsw}.

The case of $ q_j $ subject to the smallness condition of Theorem \ref{Bsw} can be obtained from this by standard methods of abstract perturbation theory.

\section{Applications}

\subsection{Smooth classes} A corollary of Theorem 1 is obtained when the conditions (i)--(iv) are satisfied with $ a_j $ independent of $ j $. In this case condition (ii) reduces to $ a_j^2 = O \( R^{d-1} \) $ and the lhs in the other three conditions is monotone decreasing in $ a_j $, hence without loss of generality one can assume that all four conditions are satisfied with $ a_j^2 = R^{ d-1 } $. Introduce the following notation. Given a finite rank Hamiltonian, $ \cG $, with parameters $ \{ e_j \} $, let $ \operatorname{Var} \cG \colon = \sum \len P_j - P_{ j+1 } \rin $, $ P_j = \langle \cdot , e_j \rangle e_j $.

\begin{corollary}\la{var}Assume $ 1/2 \le d < 1 $. Let for any $ \von > 0 $ a finite rank Hamiltonian $ \cH^\von $ defined on $ ( 0 , L ) $ exist such that

(a) \[ \len \cH - \cH^\von \rin_{ L^1 ( 0 , L ) } \le \von ,\]

(b) \[ \operatorname{Var} \cH^\von \le C \von^{ \frac{ 2d -1 }{ 2d - 2 } } . \]
Then the order of the system $ ( \cH , L ) $ is not greater than $ d $. 
\end{corollary}

\begin{proof} System $ ( \cH , L ) $ obeys the condition of Theorem 1 with $ \cH_R = \cH^{ \von ( R ) } $, $ \von ( R ) = R^{ 2 ( d-1 ) } $, $ a_j = R^{ ( d-1)/2} $. Condition (a) implies (i), (b) implies (iii) via an elementary inequality, (ii) and (iv) are immediate. \end{proof}

This corollary allows to give an upper estimate for order of Hamiltonians in classical smoothness classes. We give two examples. 

\begin{corollary}\la{holder} Let $ ( \cH , L ) $ be a canonical system with $ \operatorname{rank } \cH ( x ) = 1 $ a. e.  

1. If $ \cH \in C^\alpha [ 0 , L ] $, $ 0 < \alpha \le 1 $, then the order of the system is not greater than $ 1 - \alpha/2 $.

2. If $ \cH $ has bounded variation then the order is not greater than $ 1/2 $.
\end{corollary}

\begin{proof} The first assertion follows from choosing $ x_j = Lj/N $, $ e_j \in \Ran \cH ( x_j ) $, for the parameters of the approximating Hamiltonian $ \cH^\von $ and adjusting $ N $. The second assertion is trivial. \end{proof} 

Notice that while the estimate of Theorem 1 is sharp, the Hamiltonian in the corresponding example is discontinuous. It is an open question whether the first assertion of Corollary \ref{holder} is sharp. The second assertion of the corollary admits an "elementary" proof based on a trick from \cite[Theorem 3.6]{Teschl}. 

\subsection{Berg--Valent matrix.}\la{BVma} In this subsection we consider the order $ 1/4 $ Jacobi matrix of Berg-Valent \cite{BergValent} as a warmup for the proof of the Valent conjecture. We do not need the explicit formulae for the parameters $ q_n $ and $ \rho_n $. The necessary information about them from \cite{BergValent} is as follows

$ 1^\circ $. $ \rho_n \sim n^4  $ as $ n \to \infty $.

$ 2^\circ $. The values of the corresponding orthogonal polynomials, $ P_n ( z ) $ and $ Q_n ( z ) $, at $ z= 0 $ have asymptotics $ P_n ( 0 ) \sim c_1 n^{ -1 } $, $ Q_n ( 0 ) \sim c_2 n^{ -1 } $ with $ c_{ 1, 2 } \ne 0 $ \cite[(2.33), (3.2)]{BergValent}.

By Theorem \ref{Ka} we find that $ \delta_j \sim C j^{ -2 } $, $ \sin ( \varphi_j - \varphi_{ j+1 } ) = O \( j^{ -2 } \) $. Let $ \cH_R = \cH $ on $ ( 0 , b_{N-1} ) $ with $ N \sim R^{ 1 - d } $, $ 1/4 \le d \le 1/2 $, and define $ \cH_R $ arbitrarily on $( b_{N-1} , L ) $ so that $ \cH_R $ becomes a finite rank Hamiltonian on $ (0, L) $. Let $ a_{ N-1 } = 1 $. Then 
\bequnan \textrm{lhs of (i)} & \le & 2 ( L - b_{ N-1 } ) = 2 \sum_{ j \ge N } \delta_j = O \( R^{ d - 1 } \) , \\
\textrm{lhs of (ii)} & = & \sum_{ j=0}^{ N-2 } \frac{a_j^2}{ j^2 } + O \( R^{ d-1 } \)  , \\
\textrm{lhs of (iii)} & \le & \sum \log \( 1  +  \frac 1{ j^2 a_j a_{ j+1 }} \) + O ( 1 ). 
\eequnan 
Define $ a_j^2 = R^{ d-1 } $ for $ j \le R^d $, $ a_j^2 = R^{ 2d-1 } $ for $ R^d < j \le N-2 $. Then
$ \sum_0^{ N-2 } a_j^2 j^{-2 } = O \( R^{ d- 1 } \) $, the lhs in (iv) is $ O ( \log R ) $, and
\bequnan \sum \log \( 1  +  \frac 1{ j^2 a_j a_{ j+1 }} \) \le C R^d \log R + R^{ 1-2d } \sum_{ j > R^d } \frac 1{j^2} = O \( R^d \log R \) + O \( R^{ 1-3d } \) = \\ O \( R^d\log R \) \eequnan
because $ d \ge 1/4 $. 
Applying Theorem 1 we conclude that the order is not greater than $ 1/4 $, the actual order of the system found in \cite{BergValent}.

\subsection{Valent's conjecture} The following assertion generalizes the consideration of the previous example.

\begin{proposition}\la{1M}
Assume that a Jacobi matrix (\ref{Jac}) is such that $ P_n^2 ( 0 ) + Q_n^2 ( 0 ) \sim C n^{ \Delta - D } $, $ \rho_n \sim n^D $ as $ n \to \infty $ with numbers $ \Delta , D $ satisfying $ 1 < \Delta < D-1 $. Then the order is not greater than $ 1/D $.  
\end{proposition}

Let $ \lambda_n , \mu_n $, $ n\ge 0 $, be sequences of reals, $ \lambda_n > 0 $ for $ n \ge 0 $, $ \mu_n  > 0 $ for $ n \ge 1 $, $ \mu_0 = 0 $. Define 
\be\la{qrho}  q_{ n+1 } = \lambda_n + \mu_n , \; \rho_{ n+1} = \sqrt{ \lambda_n \mu_{ n+1 } } . \ee
The Jacobi matrix with parameters $ q_j $, $ \rho_j $ is said to be corresponding to birth-death processes with rates $ \lambda_n $ and $ \mu_n $ \cite{BergValent}.  

\begin{corollary}\la{Valenthyp}
The order of the Jacobi matrix corresponding to birth-death processes with polynomial rates $ \lambda_n = ( n + B_1 ) \cdots ( n+ B_\ell ) $, $ \mu_n = ( n + A_1 ) \cdots ( n+ A_\ell ) $ subject to the condition $ 1 < \sum ( B_j - A_j ) < \ell - 1 $, is $ 1/ \ell$.
\end{corollary}

Let us establish the proposition first. 

\begin{proof}
By Theorem \ref{Ka} the assumption of the proposition implies that 
\[ \delta_j \sim C j^{ \Delta - D } , \; \sin \( \varphi_j - \varphi_{ j+1 } \) = O \( j^{ - \Delta } \) . \]
Fix a $ d > D^{ -1 } $ small enough and define $ \cH_R $ as in Section \ref{BVma}. The value of $ N $ is to be chosen so that $ \sum_{ j \ge N } \delta_j \asymp N^{ \Delta - D + 1 } = O \( R^{ d-1 } \) $, thus let $ N \sim R^{ \frac { d-1 }{ \Delta - D + 1 }} $. Define
\[ a_j^2 = \begin{cases} R^{ d-1 } , & j \le R^d ,\cr R^{ d-1 + d ( D - \Delta - 1 ) } , & R^d < j < N-1 \cr 1 , & j = N-1 \end{cases} . \]
Notice that $ R^d \ll N-1 $, so the corresponding range of $ j $'s is non-empty. With this choice, (i) of Theorem 1 is satisfied by the choice of $ N $, (ii) is satisfied because $ \sum_{ R^d < j < N-1 }  \delta_j = O \( R^{ d ( \Delta - D + 1 ) } \) $ and $ a_j $ for this range are chosen precisely to make the corresponding term $ O  \( R^{ d-1 } \) $, and the lhs in (iv) is $ O \( \log R \) $. The lhs in (iii) is estimated above by  
\bequnan \sum \log \( 1  +  \frac 1{ j^\Delta a_j a_{ j+1 }} \) \le C R^d \log R + \frac 1{R^{ d-1 + d ( D - \Delta - 1 ) }} \sum_{ j > R^d } \frac 1{j^\Delta } = O \( R^{ d+\von } \) + \\ O \( R^{ 1 - d D + d  } \)  . \eequnan
The rhs is $ O \( R^{ d+\von } \) $ for $ d > D^{ -1 } $ and any $ \von > 0 $, and the assertion of the proposition follows by Theorem 1.
\end{proof}

\medskip

\textit{Proof of Corollary \ref{Valenthyp}}. The fact that the order does not exceed $ 1/\ell $ follows from Proposition \ref{1M} by inspection of (\ref{qrho}) and explicit formulae \cite{BergValent} expressing $ P_j ( 0 ) $, $ Q_j ( 0 ) $ via $ \lambda_j $'s and $ \mu_j $'s. On the other hand, an application of \cite[Propositions 7.1]{BergSzwarc} shows that the order is not less than $ 1/\ell $. For completeness, we provide a proof of the latter fact. First, for any $ p $ greater than the order of the system there exists a $ K > 0 $ such that for all $ z $ large enough
\[ \sum \left| P_j ( z ) \right|^2 \le e^{ K \left| z \right|^p } . \]
This is an easy corollary of the Kristoffel-Darboux formula. Since all zeroes of $ P_j $'s are real, $ | P_j ( i\tau ) | \ge \pi_j \tau^j $ for $ \tau > 0$, $ \pi_j = 1 / \( \rho_1 \cdots \rho_j \) $ being the leading coefficient of the polynomial $ P_j $, therefore 
\[ \frac 1{ \rho_j \cdots \rho_1 } \le \( \frac jC \)^{ - \frac jp} . \]
Under the assumptions of the corollary, $ \rho_j \sim j^\ell $, which implies $ p \ge 1/\ell $. 
\hfill $ \Box $  

The assertion of Corollary \ref{Valenthyp} was conjectured in \cite{Valent} on the basis of two explicitly solvable examples, the one dealt with in the previous subsection and another one \cite{GLV} with $ \ell = 3 $.

\section{Proof of Theorem \ref{selfsim}} 

The structure of the proof is as follows. First we are going to show that the order of the system is not greater than the infimum. This will be done by an application of Theorem 1 to a natural approximation $ \cH_R $. Then we will show that the order is not less than the infimum by an appropriate choice of the covering.

\textit{The order $ \le $ the infimum.} Let $ d $ be such that for some $ C > 0 $ for each $ R $ large enough there exists a covering of the interval $ ( 0 , L ) $ by $ n ( R ) \le C R^d / \log R $ intervals, to be denoted $ \omega_j  $, such that (A) is satisfied. The stated inequality will be established if we show that the order is not greater than $ d $. Without loss of generality on can assume that the intervals $ \omega_j $ are mutually disjoint. Define 
\[ \cH_R (x ) =  \begin{cases}
\frak H_1 , & x \in \omega_j,  | \omega_j \cap X_1 | \ge \frac {|\omega_j |}2 , \cr
 \frak H_2 & \mathrm{otherwise} . \end{cases} \]
With this choice of $ \cH_R $
\[ \int_{ \omega_j } \len \cH ( t ) - \cH_R ( t ) \rin \diff t \le 2 \min \{ | \omega_j \cap X_1 | , | \omega_j \cap X_2 | \} , \] 
hence the condition (i) of Theorem 1 takes the form
\[ \sum\frac 1{a_j^2} \min \{ | \omega_j \cap X_1 | , | \omega_j \cap X_2 | \} \le C R^{ d-1 } . \]
Notice that $ \min \{ | \omega_j \cap X_1 | , | \omega_j \cap X_2 | \} \asymp | \omega_j \cap X_1 | \, | \omega_j \cap X_2 | / | \omega_j | $, hence the latter condition is equivalent to
\be\la{i} \sum\frac 1{a_j^2 | \omega_j | } | \omega_j \cap X_1 |\,  | \omega_j \cap X_2 | \le C R^{ d-1 } . \ee 
Conditions (iii) and (iv) of Theorem 1 in the situation under consideration are satisfied if 
\be\la{iii} \sum \log \( 1 + a_j^{ -1 } \) \le C R^d , \ee
and condition (ii) has the form
\be\la{ii} \sum a_j^2 | \omega_j | \le C R^{ d -1 } . \ee

Let $ a_j = 1 $ for $ | \omega_j | \le 2/R $. We write $ \frak N = \{ j \colon | \omega_j | \le 2/R \} $. The parts of sums in (\ref{i}), (\ref{iii}) and (\ref{ii}) over $ j \in \frak N $ are then estimated above by $ n ( R ) R^{ -1 } $, $ n ( R ) R^{ -1 } $ and $ n( R ) $, respectively, hence they are $ O \( R^{ d-1 } \) $ by condition (B). 
For $ | \omega_j | > 2/R $ we optimize the choice of $ a_j $ over the summands in (\ref{i}), (\ref{iii}) and (\ref{ii}) by taking 
\[ a_j^2 = \max \left\{ \frac 1{ R | \omega_j | },  \frac{ \sqrt{ | \omega_j \cap X_1 | | \omega_j \cap X_2 | }}{ | \omega_j | } \right\}  . \]

With this choice the sums over $ j \notin \frak N $ in (\ref{i}), (\ref{ii}) and (\ref{iii}) are estimated above by \[ \sum \sqrt{ | \omega_j \cap X_1 | \, | \omega_j \cap X_2 | } ,\]  
\[ R^{ -1 } n( R ) + \sum \sqrt{ | \omega_j \cap X_1 | \, | \omega_j \cap X_2 | } , \]
and 
\[ \sum \log ( R | \omega_j | ) \le n ( R ) \log R , \]
respectively. Plugging here (A) and (B) of Theorem 2 we obtain that all the assumptions of Theorem 1 are satisfied. 

\textit{The infimum $ \le $ the order.} Given a $ \tau > 0 $, $ x \in ( 0 , L ) $, let $ s ( \tau , x ) \in [ 0 , x ] $ be the solution of the equation $ \tau^2 | ( s, x ) \cap X_1 | \, | ( s, x ) \cap X_2 | = 1 $. This solution is unique whenever exists. Without loss of generality one can assume that, say, for some $ a > 0 $ the function $ \cH ( x ) = \frak H_1 $ for $ x \in ( 0 , a/2  ) $, $ \cH ( x ) = \frak H_2 $ for $ x \in ( a/2  , a) $ (attaching such two intervals at the left end does not change the order). Then $ s ( \tau , x ) $ is defined for $ \tau $ large enough for all $ x \ge a $.

\begin{lemma}\la{Katzform}\cite[lemmas 1--3]{Katz1} The order of the system $ ( \cH , L ) $ equals to
\be\la{ordKa} \limsup_{ \tau \to +\infty } \frac{ \displaystyle{\int_a^L} \frac{\displaystyle{\chi_2 ( x ) }}{ \displaystyle{| ( s ( \tau , x ) , x )\cap X_2 | }} \diff x }{ \log \tau }  , \ee
$ \chi_2 $ being the indicator function of the set $ X_2 $.
\end{lemma} 

This assertion provides a crucial step in the proof of the Kats' formula for the order. It is formulated in \cite{Katz1} in terms of the corresponding strings. For completeness we reproduce here the proof of the part of the lemma that we are going to use -- the order is not less than the quantity (\ref{ordKa}), translated to the language of canonical systems. 

\begin{proof} By definition the order of a system is the order of any of the matrix elements of the monodromy matrix. Let us show that the order of the matrix element $ M_{ 11 } ( z ) $ is estimated from below by the rhs in (\ref{ordKa}). The order of $ M_{ 11 } $ coincides with 
\[ \limsup_{ \tau \to + \infty } \frac{\log \log | M_{ 11 } ( i \tau ) |}{ \log \tau } \] 
because $ M_{ 11 } $ is a real entire function having all its zeroes real.

Let $ \chi_1 = 1 - \chi_2 $, $ \rho ( s , t ) = | ( s , t ) \cap X_2 | $. On rewriting the first column of (\ref{can}) as an integral equation we obtain,
\[ M_{11} ( z, x ) = 1 + z \int_0^x \chi_2 ( t ) M_{ 21} ( z , t ) \diff t = 1 - z^2 \int_0^x \chi_1 ( t ) \rho ( t , x ) M_{11} ( z , t ) \diff t . \]  
When $ z = i \tau $, $ \tau > 0 $, this becomes ($ \xi_\tau ( x ) \colon = M_{ 11 } ( i \tau , x ) $),
\[  \xi_\tau ( x )  = 1 + \tau^2 \int_0^x \chi_1 ( t ) \rho ( t , x ) \xi_\tau (t) \diff t . \]
It shows that $ \xi_\tau $ is a positive and monotone nondecreasing function. Let us estimate $ \xi_\tau^\prime ( x ) / \xi_\tau ( x ) $ from below. For a. e. $ x \in X_2 $ and all $ s \le x $, we have
\bequnan \xi_\tau ( s ) = \xi_\tau ( x ) - \tau^2 \rho ( s , x ) \int_0^s \chi_1 ( t ) \xi_\tau \diff t -  \tau^2 \ \int_s^x  \rho ( t , x ) \chi_1 ( t ) \xi_\tau ( t ) \diff t \\ \ge  \xi_\tau ( x ) - \tau^2 \rho ( s , x ) \int_0^s \chi_1  \xi_\tau \diff t -  \tau^2 \rho ( s , x ) \int_s^x  \chi_1 \xi_\tau \diff t = \xi_\tau ( x ) - \rho ( s , x ) \xi_\tau^\prime ( x ) ,\eequnan
and thus
\bequnan  \xi_\tau^\prime ( x ) = \tau^2 \( \int_s^x +  \int_0^s \) \chi_1 \xi_\tau \diff t\ge \tau^2 \int_s^x \chi_1 \xi_\tau \diff t \ge \tau^2 \left| ( s , x ) \cap X_1 \right| \, \xi_\tau ( s ) \\ \ge \tau^2  \left| ( s , x ) \cap X_1 \right| \( \xi_\tau ( s ) - \rho ( s , x ) \xi_\tau^\prime ( x ) \)  . 
\eequnan
Picking $ s = s ( \tau , x ) $ we obtain that for a. e. $ x \in X_2 \cap ( a ,  L )  $
\[ \frac{\xi_\tau^\prime ( x ) }{\xi_\tau ( x )}\ge \frac 1{2 \rho ( s ( \tau , x )  , x ) } , \] 
which implies the required assertion upon integration in $ x $ over $ [ a , L ] $.
\end{proof} 

The proof of Theorem 2 will be completed if we show that for any $ d $ such that 
\be\la{condKatz} \int_a^L \frac{\chi_2 ( x ) }{ \rho ( s ( \tau , x ) , x ) } \diff x = O \( \tau^d \) , \; \tau \to + \infty ,\ee
the interval $ ( 0 , L ) $ can be covered by $ O \( R^d \log R \) $ intervals, $ \omega_j $, so that 
\be\la{sqrt} \sum \sqrt{ | \omega_j \cap X_1 | \, | \omega_j \cap X_2 | } = O \(  R^{ d-1 } \) . \ee
For each $ R $ large enough define a monotone decreasing sequence $ \{ x_j \} $, $ j \ge 1 $, as follows, $ x_1 = L $, $ x_{ j+1 } = s ( R, x_j ) $, if $ j \ge 1 $ and $ x_j \ge a $; if $ x_{ j-1 } \ge a $, $ x_j < a $ then $ x_{ j+1 } = 0 $ and the sequence terminates. Observe that the sequence $ x_j $ is finite. This follows from the definition of the function $ s ( R , x ) $, for 
\[ 1 =  R^2 | ( s ( R , x ) , x ) \cap X_1 |  \, | ( s ( R , x ), x ) \cap X_2 | \le R^2 \left| x - s ( R , x ) \right|^2 , \]
which means that $ x_j - x_{ j+1 } \ge R^{ -1 } $ so the sequence has $ O ( R ) $ members. Define $ \omega_j = [ x_{ j+1 }, x_j ] $. By construction $ [ 0 , L ] = \cup_j \omega_j $. We claim that this is the required covering. First, we have to show that $ N (R) $, the number of intervals in the covering, is $ O \( R^d \log R \) $. To this end,
notice that $  \rho ( s ( \tau , x ) , x ) \le \rho ( x_{ j+2 } , x_j ) $ for $ x \in [ x_{ j+1 } , x_j ] $ hence 
\[ \textrm{lhs in (\ref{condKatz})} \ge \sum_{ j=1}^{ N( R )-1 } \frac { s_j }{ s_j + s_{ j+1 } } , \; s_j \colon = \rho ( x_{ j+1 } , x_j ) . \] 
Let $ \frak g = \{ j \colon \frac{ s_{ j+1 }}{ s_j } \le 2 \} $, $ \hat{\frak g} = \{  j \colon \frac{ s_{ j+1 }}{ s_j } \ge 2 \} $; $ n_{ \frak g } $, $ {\hat n}_{ \frak g} $ be the respective numbers of elements. When $ j \in \frak g $ the summand in the displayed sum is bounded below, hence, $ n_{ \frak g } $ is $ O ( R^d ) $ by (\ref{condKatz}). To estimate $ {\hat n}_{ \frak g } $ notice that $ s_j \ge 1/ \( L R^2 \) $, for $ 1 =  R^2 \left| \omega_j \cap X_1 \right| s_j \le R^2 L s_j $. It follows that if $ k $ is the length of a discrete interval of the set $ \hat{ \frak  g } $ and $ m $ is the right end of it then $x_m -  x_{ m+1 } \ge 2^{ k-1 } / \( L R^2 \) $. On the other hand, $ x_m - x_{ m+1 } \le L $ trivially, hence $ k \le C + 2 \log_2 R $, and therefore $ {\hat n}_{ \frak g } = O \(R^d \log R \) $. Thus, $ N ( R ) = n_{ \frak g } + {\hat n}_{ \frak g} = O \(R^d \log R \)  $ as required. To complete the proof, notice that by the very definition of $ s ( R ,x ) $, the summand in (\ref{sqrt}) equals to $ R^{ -1} $, hence (\ref{sqrt}) reduces to $ N ( R ) = O \( R^d \) $ and thus holds trivially. \hfill $ \Box $

\section{Comments on Theorem 2 and applications}


\subsection{The Cantor string}\la{Castr} Let $ \xi \colon [ 0 , 1 ] \to [ 0 , 1 ] $ be the standard Cantor function, $ ( u_j , v_j ) $ be its constancy intervals, $ T ( x ) = x + \xi ( x ) $, $ L = 2 $. Define the Hamiltonian $ \cH $ on $ [ 0 , 2 ] $ to be 
\[ \cH ( x ) = \begin{cases} \frak H_1 , & x \in \bigcup_j T \(  \left[ u_j , v_j \right] \) \cr \frak H_2 & \textrm{otherwise.} \end{cases} . \] 
The canonical system $ ( \cH , 2 ) $ is called the Cantor string. We are going to show that  the assumptions of Theorem \ref{selfsim} hold for $ d = d_C = 2 / \log_2 6 $. Let $ \tau_j $ be the union of $ 2^{ j - 1 } $ intervals thrown away on the $ j $-th step of construction of the Cantor set. Define $ M_R = T \( \bigcup_{ k=1 }^j \tau_k \) $, $ j \sim d \log_2 R $. The set $ M_R $ is a union of $ O \( 2^j \) $ non-intersecting intervals. Consider the covering of $ [ 0 , 2 ] $ by the intervals of the set $ M_R $ and their contiguency intervals, the latter to be denoted $ \omega_j $. By construction, the overall number of intervals in this covering is $ O \( 2^j \) = O \( R^d \) $.

The terms corresponding to the intervals of the set $ M_R $ in the sum in condition (A) obviously vanish, hence the sum reduces to 
\[  \sum \sqrt{ | \omega_j \cap X_1 | \, | \omega_j \cap X_2 | } \le \sqrt{ \sum | \omega_j \cap X_1 | } = \sqrt{ \left| \( (0,2) \setminus M_R \) \cap X_1  \right| } \] 
Since $ T $ is linear on intervals of sets $ \tau_k  $, $ T^{ - 1} X_1 = \bigcup_j \tau_j $, and therefore
\[ \left| \( (0,2) \setminus M_R \) \cap X_1 \right| = \sum_{ k > j } | \tau_k | = \( \frac 23 \)^j  \sim R^{ d ( 1 - \log_2 3  ) } = R^{ 2 ( d-1 ) }  . \]
It follows that the assumption (A) is satisifed. Applying Theorem \ref{selfsim} we conclude that the order of this system is not greater than $ d_C $. Let us show that the order is not less than $ d_C $. Assume that the condition of Theorem 2 is satisfied for some $ d $ and let us use $ \frak D = \frak D ( R ) = \{ \Delta_j \} $ for the covering whose existence is required by the condition. Let $ j = d \log_2 R $ and $ M_R $ be defined as above. Then without loss of generality one can assume that the intervals $ \Delta_j $ are mutually disjoint and the intervals of $ M_R $ are among them. Indeed, adding the intervals of $ M_R $ to $ \frak D $ and removing the intersections of intervals of $ \frak D $ with the intervals of $ M_R $ does not increase the lhs in (A) and increases the constant in the rhs of (B) by at most $ 1 $. Given an $ \omega $, an interval of contiguity of $ M_R $, let $ n_{ \omega , R } $ be the number of intervals of $ \frak D $ belonging to interval $ \omega $, then $ \sum_\omega n_{ \omega , R } \le C R^d = C 2^j $. Since the number of intervals of contiguity is $ 2^j $, it follows that the number of $ k $'s for which $ n_{ k , R } \le 2C $ is greater than $ 2^{ j-1 } $. Let $ \omega $ be a contiguity interval for which $  n_{ k , R } \le 2C $, $ j_0 = j + \log_2 ( 2 C ) + 4 $. By a Dirchlet box argument then there exists an interval $ \Delta \in \frak D $ which contains two nearby intervals of $ T \( \tau_{ j_0} \) $, for the number of intervals of $ T \( \tau_{ j_0 } \) $ contained in $ \omega $ is $ 2^{ j_0 - j } $. By construction the measure $ | \Delta \cap X_1 | \ge 2 \cdot 3^{ - j_0 } = C 3^{ -j } $, and $ | \Delta \cap X_2 | \ge 2^{ - j_0 } = C 2^{ -j } $, so $ \sqrt{ | \Delta \cap X_1 | | \Delta \cap X_2 | } \ge C 2^{ -j/2 } 3^{ -j/2 } $. Since the number of such intervals $ \Delta $ is estimated below by $ C 2^j $ we find that the lhs in (A) is estimated below by $ C \( 2/3 \)^{ j / 2 } = C R^{ d \( 1 - \log_2 3 \)/2 } $. This being $ O \( R^{ d-1 } \) $ means that $ d \ge d_C $. Thus we have derived that the order equals to $ d_C $, the result obtained\footnote{The factor $ 2 $ in the numerator in the expression for $ d_C $ is due to the spectral parameter in our definition of the string being the square root of a "natural" parameter used in \cite{UH}.} in \cite{UH} or \cite{SolVerb} by other means. 



\subsection{Kats formula}\la{Katzf} As mentioned in the introduction a direct comparison of Theorem 2 and Kats formula (\ref{Katzformula}) is not possible for lack of examples using the latter in the situation of the former.  Notice however that, properly understood, (\ref{Katzformula}) holds for a class of problems with $ L = \infty $ (singular strings) and examples are known \cite{Katz2} in this class where the order is calculated by application of (\ref{Katzformula}).

\bigskip

\noindent\textbf{Acknowledgements.} The author is indebted to H. Woracek for attracting his attention to the order problem and useful remarks, to I. Sheipak for references, and to a referee for suggested improvements of the presentation.
This work was supported in part by the Austrian Science Fund (FWF) project I 1536--N25, and the Russian Foundation for Basic Research, Grants 13-01-91002-ANF and 12-01-00215.

\end{document}